\documentclass[11pt]{amsart}

\usepackage{amsmath}
\usepackage{amsfonts}
\usepackage[active]{srcltx}

\newtheorem{theorem}{Theorem}

\newtheorem{lemma}{Lemma}

\newtheorem{definition}{Definition}
\newtheorem{corollary}{Corollary}
\newtheorem*{TDW}{Theorem DW}

\numberwithin{equation}{section}

\begin{document}
{On the convergence of double Fourier series of functions of bounded partial generalized  variation.}

\medskip
{Ushangi Goginava and Artur Sahakian}

\vskip1cm
{ABSTRACT.}
The convergence of double Fourier series of functions of bounded partial $\Lambda$-variation is investigated. The sufficient and necessary conditions on the sequence $\Lambda=\{\lambda_n\}$ are found for the convergence of  Fourier series of functions of bounded partial $\Lambda$-variation.

\vskip2cm
{U. Goginava, Institute of Mathematics, Faculty of Exact and Natural
Sciences, Tbilisi State University, Chavchavadze str. 1, Tbilisi 0128,
Georgia}

{z\_goginava@hotmail.com}

\medskip
{A. Sahakian, Erevan State University, Faculty of Mathematics and Mechanics,
Alex Manoukian str. 1, Yerevan 0025, Armenia}

{sart@ysu.am}
\medskip
\medskip

\vfill
\eject
{On the convergence of double Fourier series}

Artur Sahakian

sart@ysu.am
\eject

\section{Classes of Functions of Bounded Generalized  Variation}
In 1881 Jordan \cite{Jo} introduced a class of functions of bounded variation and applied it to the theory of Fourier series. Hereinafter this notion was generalized by many authors (quadratic variation, $\Phi$-variation, $\Lambda$-variation ets., see \cite{Wi}-\cite{Ch}). In two dimensional case the class BV of functions of bounded variation was introduced by Hardy \cite{Ha}.

Let $f$ be a real function of two variable of period $2\pi $ with respect to
each variable. Given intervals $I=(a,b)$, $J=(c,d)$ and points $x,y$ from
$T:=[0,2\pi ]$ we denote
\begin{equation*}
f(I,y):=f(b,y)-f(a,y),\qquad f(x,J)=f(x,d)-f(x,c)
\end{equation*}
and
\begin{equation*}
f(I,J):=f(a,c)-f(a,d)-f(b,c)+f(b,d).
\end{equation*}
Let $E=\{I_{i}\}$ be a collection of nonoverlapping intervals from $T$ ordered in
arbitrary way and let $\Omega $ be the set of all such collections $E$. Denote
by $\Omega _{n}$ set of all collections of $n$ nonoverlapping intervals $%
I_{k}\subset T.$

For the sequence of positive numbers $\Lambda =\{\lambda
_{n}\}_{n=1}^{\infty }$ we denote
\begin{equation*}
\Lambda V_{1}(f)=\sup_{y}\sup_{E\in \Omega }\sum_{n}\frac{|f(I_{i},y)|}{%
\lambda _{i}}\,\,\,\,\,\,\left( E=\{I_{i}\}\right) ,
\end{equation*}
\begin{equation*}
\Lambda V_{2}(f)=\sup_{x}\sup_{F\in \Omega }\sum_{m}\frac{|f(x,J_{j})|}{%
\lambda _{j}}\qquad (F=\{J_{j}\}),
\end{equation*}
\begin{equation*}
\Lambda V_{1,2}(f)=\sup_{F,\,E\in \Omega }\sum_{i}\sum_{j}\frac{%
|f(I_{i},J_{j})|}{\lambda _{i}\lambda _{j}}.
\end{equation*}

\begin{definition}
We say that the function $f$ has Bounded $\Lambda $-variation on $T=[0,2\pi]^2$ and write $%
f\in \Lambda BV$, if
\begin{equation*}
\Lambda V(f):=\Lambda V_{1}(f)+\Lambda V_{2}(f)+\Lambda V_{1,2}(f)<\infty .
\end{equation*}
We say that the function $f$ has Bounded Partial $\Lambda $-variation and
write $f\in P\Lambda BV$ if
\begin{equation*}
P\Lambda V(f):=\Lambda V_{1}(f)+\Lambda V_{2}(f)<\infty .
\end{equation*}
\end{definition}

If $\lambda_n\equiv 1$ (or if $0<c<\lambda_n<C<\infty,\ n=1,2,\ldots$) the
classes $\Lambda BV$ and $P\Lambda BV$ coincide with the Hardy class $BV$ and PBV respectively. Hence it is
reasonable to assume that $\lambda_n\to\infty$ and since the intervals in
$E=\{I_i\}$ are ordered arbitrarily, we will suppose, without
loss of generality, that the sequence $\{\lambda_n\}$ is increasing. Thus,
\begin{equation}\label{Lambda}
1<\lambda_1\leq \lambda_2\leq\ldots,\qquad \lim_{n\to\infty}\lambda_n=\infty.
\end{equation}

In the case when $\lambda _{n}=n,\ n=1,2\ldots $ we say \textit{Harmonic
Variation} instead of $\Lambda $-variation and write $H$ instead of $\Lambda$ ($HBV$, $PHBV$, $HV(f)$, ets).

The notion of $\Lambda$-variation was introduced by
D. Waterman \cite{W} in one dimensional case and A. Sahakian \cite{Saha} in two dimensional case.
\begin{definition}
Let $\Phi $-be a strictly increasing continuous function on $[0,+\infty )$
with $\Phi \left( 0\right) =0$. We say that the function $f$ has bounded
partial $\Phi $-variation on $T^{2}$ and write $f\in PBV_\Phi$, if
\begin{equation*}
V_{\Phi }^{\left( 1\right) }\left( f\right)
:=\sup\limits_{y}\sup\limits_{\{I_{i}\}\in \Omega
_{n}}\sum\limits_{i=1}^{n}\Phi \left( |f\left( I_{i},y\right) |\right)
<\infty ,\quad n=1,2,...,
\end{equation*}
\begin{equation*}
V_{\Phi }^{\left( 2\right) }\left( f\right)
:=\sup\limits_{x}\sup\limits_{\{J_{j}\}\in \Omega
_{m}}\sum\limits_{j=1}^{m}\Phi \left( |f\left( x,J_{j}\right) |\right)
<\infty ,\quad m=1,2,....
\end{equation*}
\end{definition}

In the case when $\Phi \left( u\right) =u^{p},\,p\geq 1$, the notion of bounded
partial $p$-variation (class $PBV_p$) was introduced in \cite{GoJAT}.

\begin{theorem}\label{T1}
Let $\Lambda =\{\lambda _{n}=n\gamma_n\}$ and $\gamma_{n}\geq \gamma _{n+1}>0,\ n=1,2,....\,\,\,$. \\
1) If
\begin{equation}\label{T1-1}
\sum_{n=1}^{\infty }\frac{\gamma _{n}}{n}<\infty,
\end{equation}
then $P\Lambda BV\subset HBV$.\\
2) If, in addition, for some $\delta>0$
\begin{equation}\label{T1-10}
\gamma_n=O(\gamma_{n^{[1+\delta]}})\quad \text{as}\quad n\to \infty
\end{equation}
and
\begin{equation}\label{T1-11}
\sum_{n=1}^{\infty }\frac{\gamma _{n}}{n}=\infty,
\end{equation}
then $P\Lambda BV\not\subset HBV$.
\end{theorem}

\textbf{Proof.} 1) Let $f\in P\Lambda BV$ and
\begin{equation*}
\sum\limits_{i,j=1}^{\infty }\frac{|f\left( I_{i},J_{j}\right) |}{ij}%
=\sum\limits_{i\leq j}\frac{|f\left( I_{i},J_{j}\right) |}{ij}%
+\sum\limits_{i>j}\frac{|f\left( I_{i},J_{j}\right) |}{ij}:=I_{1}+I_{2}.
\end{equation*}
Then according to (\ref{T1-1}),
\begin{eqnarray*}
I_{1} &=&\sum\limits_{i=1}^{\infty }\frac{1}{i}\sum\limits_{j=i}^{\infty }%
\frac{|f\left( I_{i},J_{j}\right) |}{j} \\
&\leq &2\sum\limits_{i=1}^{\infty }\frac{1}{i}\sup\limits_{x}\sum%
\limits_{j=i}^{\infty }\frac{|f\left( x,J_{j}\right) |}{\lambda _{j}}\frac{%
\lambda _{j}}{j} \\
&\leq &2\sum\limits_{i=1}^{\infty }\frac{\lambda _{i}}{i^{2}}%
\sup\limits_{x}\sum\limits_{j=i}^{\infty }\frac{|f\left( x,J_{j}\right) |}{%
\lambda _{j}} \\
&\leq &2\Lambda V_{2}\left( f\right) \sum\limits_{i=1}^{\infty }\frac{%
\lambda _{i}}{i^{2}}\leq c\Lambda V_{2}\left( f\right) <\infty .
\end{eqnarray*}
Similary, $I_{2}\leq c\Lambda V_{1}\left( f\right)<\infty .$

2) In the proof of the second statement of Theorem \ref{T1} we use the following well known lemma .
\begin{lemma}\label{L1}
Let $u_i$ and $v_i$, $i=1,2,\ldots,j$ be two increasing
(decreasing) sequences of positive numbers. Then for any rearrangement $\{\sigma(i)\}$ of the set $\{1,2,\ldots,j\}$
$$
\sum_{i=1}^ju_iv_{j-i+1}\leq \sum_{i=1}^ju_iv_{\sigma(i)}
\leq \sum_{i=1}^ju_iv_i.
$$
\end{lemma}

Let (\ref{T1-10}) and (\ref{T1-11}) be fulfilled and define
\begin{equation*}
f( x,y) :=
\begin{cases}
t_j,& x=\frac1i,\  y=\frac1j,\ j<i\leq j+m_j,\ i,j=1,2,... \\
0,&\text{otherwise}
\end{cases},
\end{equation*}
where
\begin{equation}\label{tj}
t_j:=\left(\sum_{i=1}^{m_j}\frac1{\lambda_j}
\right)^{-1},\qquad m_j=\left[j^{1+\delta}\right],\quad j=1,2,\ldots
\end{equation}

Let $x=1/i$ and let $j(i)$ be the smallest integer satisfying
\begin{equation}\label{ji}
j(i)+m_{j(i)}\geq i.
\end{equation}


Since $t_j$ is decreasing and $\lambda_j$ is increasing, using Lemma \ref{L1} we can write
\begin{eqnarray*}
&&\sup\limits_{F\in \Omega }\sum\limits_{j=1}^{\infty }\frac{|f\left(
1/i,J_{j}\right) |}{\lambda _{j}} \\
&= &\sum_{j=j(i)}^{i-1}
\frac{{t_{j}}}{\lambda_{j-j(i)} } \leq
t_{j(i)}\sum_{j=1}^{i-j(i)}
\frac{1}{\lambda_j}\leq
t_{j(i)}\sum_{j=1}^{m_{j(i)}}\frac{1}{\lambda_j}=1.
\end{eqnarray*}
Hence
\begin{equation}\label{4}
\Lambda V_{2}(f)\leq1 .
\end{equation}

For $y=1/j$ we have
$$
\sup\limits_{E\in \Omega }\sum\limits_{i=1}^{\infty }\frac{|f\left(I_{i},1/j\right) |}{\lambda _{i}}
=
t_{j}\sum_{i=1}^{m_{j}}\frac{1}{\lambda_i}=1.
$$
Consequently,
\begin{equation}\label{5}
\Lambda V_{1}(f)\leq1 .
\end{equation}

Combining (\ref{4}) and (\ref{5}) we conclude that $f\in P\Lambda BV.$

Now we prove that $f\not\in HBV.$
From (\ref {T1-10}) and (\ref{tj})follows that
$$
\sum_{i=1}^{m_j}\frac1{\lambda_i}=
\sum_{i=1}^{m_j}\frac1{i\gamma_i}\leq
C\frac{\log m_j}{\gamma_{m_j}}\leq
C\frac{\log j}{\gamma_{j}}.
$$
Hence
\begin{equation}\label{tjl}
t_j\cdot{\log j}\geq c{\gamma_j},\qquad j=2,3,\ldots
\end{equation}
and from the definition of $f$, (\ref{tj}) and (\ref{T1-11}) we obtain
\begin{eqnarray*}
&&\sup\limits_{E,F\in \Omega }\sum\limits_{i,j}\frac{|f\left(
I_{i},J_{j}\right) |}{ij} \\
&\geq &\sum\limits_{j=1}^{\infty }
\frac {t_j}j
\sum\limits_{i=j+1}^{j+m_j}\frac1i
\geq
c\sum\limits_{j=1}^{\infty }
\frac {t_j}j\log(j+m_j)
\geq
c\sum\limits_{j=1}^{\infty }
\frac {\gamma_j}j=\infty .
\end{eqnarray*}

Theorem 1 is proved.

Taking $\lambda_n\equiv 1$ \, and $\lambda_n=n$ \,in Theorem 1, we get

\begin{corollary}
$PBV\subset HBV$ and $PHBV\not\subset HBV$.
\end{corollary}

\begin{corollary}
Let $\Phi $ and $\Psi $ are conjugate functions in the sense of Yung ($%
ab\leq \Phi (a)+\Psi (b)$) and let for some $\{\lambda _{n}\}$ satisfying (1)
\begin{equation}\label{Phi}
\sum_{n=1}^{\infty }\Psi \left( \frac{1}{\lambda _{n}}\right) <\infty .
\end{equation}
Then $PBV_{\Phi}\subset HBV$.
In particular, $PBV_{p}\subset HBV$ for any $p>1$.
\end{corollary}
Indeed,
from the inequality $\frac{a}{\lambda}\leq \Phi(a)+\Psi( \frac{1}{\lambda})$ follows that
 $PBV_{\Phi}\subset P\Lambda BV$  under assumption (\ref{Phi}), and $P\Lambda BV\subset HBV$ if (\ref{Lambda}) holds.

\begin{definition}
[see \cite{GoEJA}] The partial modulus of variation of a function $f$ are
the functions $v_{1}\left( n,f\right)$ and
$ v_{2}\left( m,f\right) $ defined by
\begin{equation*}
v_{1}\left( n,f\right) :=\sup\limits_{y}\sup\limits_{\{I_{i}\}\in \Omega
_{n}}\sum\limits_{i=1}^{n}\left| f\left( I_{i},y\right) \right| ,\quad n=1,2,\ldots,
\end{equation*}
\begin{equation*}
v_{2}\left( m,f\right) :=\sup\limits_{x}\sup\limits_{\{J_{k}\}\in \Omega
_{m}}\sum\limits_{i=1}^{m}\left| f\left( x,J_{k}\right) \right| ,\quad m=1,2,\ldots.
\end{equation*}
\end{definition}

For functions of one variable  the concept
of the modulus variation was introduced by Chanturia \cite{Ch}.


\begin{theorem}
If $f\in B$ is bounded on $T^2$ and
\begin{equation*}
\sum\limits_{n=1}^{\infty }\frac{\sqrt{v_{j}\left( n,f\right) }}{n^{3/2}}%
<\infty ,\quad j=1,2,
\end{equation*}
then $f\in HBV.$
\end{theorem}

\textbf{Proof.} Using Abel transformation we can write
\begin{eqnarray*}
\sum\limits_{k=1}^{m}\frac{|f\left( x,J_{k}\right) |}{k} &=&\sum%
\limits_{k=1}^{m-1}\left( \frac{1}{k}-\frac{1}{k+1}\right)
\sum\limits_{l=1}^{k}|f\left( x,J_{l}\right) |+\frac{1}{m}%
\sum\limits_{k=1}^{m}|f\left( x,J_{k}\right) \\
&\leq &\sum\limits_{k=1}^{m-1}\frac{1}{k^{2}}\left(
\sum\limits_{l=1}^{k}|f\left( x,J_{l}\right) |\right) ^{1/2}\left(
\sum\limits_{l=1}^{k}|f\left( x,J_{l}\right) |\right) ^{1/2}+c \\
&\leq &c\sum\limits_{k=1}^{m-1}\frac{\sqrt{k}}{k^{2}}\left(
\sum\limits_{l=1}^{k}|f\left( x,J_{l}\right) |\right) ^{1/2}+c \\
&\leq &c\sum\limits_{k=1}^{\infty }\frac{\sqrt{v_{2}\left( k,f\right) }}{%
k^{3/2}}+c\leq c<\infty .
\end{eqnarray*}
Consequently,
\begin{equation}\label {6}
HV_{2}\left( f\right) <\infty .
\end{equation}

Analogously, we can prove that
\begin{equation}
HV_{1}\left( f\right) <\infty .
\end{equation}

Using Hardy transformation we obtain
\begin{eqnarray}
&&\sum\limits_{i=1}^{n}\sum\limits_{j=1}^{m}\frac{|f\left(
I_{i},J_{j}\right) |}{ij} \notag \\
&=&\sum\limits_{i=1}^{n-1}\sum\limits_{j=1}^{m-1}\left( \frac{1}{i}-\frac{1}{%
i+1}\right) \left( \frac{1}{j}-\frac{1}{j+1}\right)
\sum\limits_{l=1}^{i}\sum\limits_{s=1}^{j}|f\left( I_{l},J_{s}\right) |\notag
\\
&&+\frac{1}{n}\sum\limits_{j=1}^{m-1}\left( \frac{1}{j}-\frac{1}{j+1}\right)
\sum\limits_{l=1}^{n}\sum\limits_{s=1}^{j}|f\left( I_{l},J_{s}\right) |
\\
&&+\frac{1}{m}\sum\limits_{i=1}^{n-1}\left( \frac{1}{j}-\frac{1}{j+1}\right)
\sum\limits_{l=1}^{i}\sum\limits_{s=1}^{m}|f\left( I_{l},J_{s}\right) |
\notag \\
&&+\frac{1}{nm}\sum\limits_{i=1}^{n}\sum\limits_{j=1}^{m}|f\left(
I_{i},J_{j}\right) |  \notag \\
&=&I+II+III+IV.  \notag
\end{eqnarray}

Since
\begin{equation*}
\sum\limits_{l=1}^{i}\sum\limits_{s=1}^{j}|f\left( I_{l},J_{s}\right) |\leq
2i\sup\limits_{x}\sum\limits_{s=1}^{j}|f\left( x,J_{s}\right) |\leq
2iv_{2}\left( j,f\right)
\end{equation*}
and
\begin{equation*}
\sum\limits_{l=1}^{i}\sum\limits_{s=1}^{j}|f\left( I_{l},J_{s}\right) |\leq
2j\sup\limits_{y}\sum\limits_{l=1}^{i}|f\left( I_{l},y\right) |\leq
2jv_{1}\left( i,f\right)
\end{equation*}
we can write
\begin{eqnarray}
I &\leq &\sum\limits_{i=1}^{n-1}\sum\limits_{j=1}^{m-1}\frac{1}{i^{2}j^{2}}%
\left( \sum\limits_{l=1}^{i}\sum\limits_{s=1}^{j}|f\left( I_{l},J_{s}\right)
|\right) ^{1/2}\left( \sum\limits_{l=1}^{i}\sum\limits_{s=1}^{j}|f\left(
I_{l},J_{s}\right) |\right) ^{1/2} \notag \\
&\leq &2\sum\limits_{i=1}^{n-1}\sum\limits_{j=1}^{m-1}\frac{\sqrt{%
ijv_{2}\left( j,f\right) v_{1}\left( i,f\right) }}{i^{2}j^{2}}  \\
&\leq &2\sum\limits_{i=1}^{\infty }\frac{\sqrt{v_{1}\left( i,f\right) }}{%
i^{3/2}}\sum\limits_{j=1}^{\infty }\frac{\sqrt{v_{2}\left( j,f\right) }}{%
j^{3/2}}<\infty ,  \notag
\end{eqnarray}
\begin{eqnarray}
II &\leq &\frac{1}{n}\sum\limits_{j=1}^{m-1}\frac{1}{j^{2}}\left(
\sum\limits_{l=1}^{n}\sum\limits_{s=1}^{j}|f\left( I_{l},J_{s}\right)
|\right) ^{1/2}\left( \sum\limits_{l=1}^{n}\sum\limits_{s=1}^{j}|f\left(
I_{l},J_{s}\right) |\right) ^{1/2}\notag \\
&\leq &\frac{1}{n}\sum\limits_{j=1}^{m-1}\frac{\sqrt{njv_{2}\left(
j,f\right) }}{j^{2}}   \\
&\leq &\frac{\sqrt{v_{1}\left( n,f\right) }}{\sqrt{n}}\sum\limits_{j=1}^{%
\infty }\frac{\sqrt{v_{2}\left( j,f\right) }}{j^{3/2}}\leq c<\infty
,n=1,2,...,  \notag
\end{eqnarray}
Analogously, we can prove that
\begin{equation}
III\leq c<\infty ,
\end{equation}
\begin{equation}\label {12}
IV\leq 2\sqrt{\frac{v_{1}\left( n,f\right) }{n}\frac{v_{2}\left( m,f\right)
}{m}}\leq c<\infty ,\,\,n,m=1,2,....
\end{equation}

Combining (\ref{6})-(\ref{12}) we conclude that $f\in HBV.$ Theorem 2 is proved.

\section{Convergence of two-dimensional trigonometric Fourier series}

Let $f\in L^{1}\left( T^{2}\right),\ T^{2}:=\left[ 0,2\pi \right]^2.$ The Fourier series of $f$ with respect to the
trigonometric system is the series
\begin{equation*}
S \left[ f\right] :=\sum_{m,n=-\infty }^{+\infty }\widehat{f}\left(
m,n\right) e^{imx}e^{iny},
\end{equation*}
where
\begin{equation*}
\widehat{f}\left( m,n\right) =\frac{1}{4\pi ^{2}}\int_{0}^{2\pi
}\int_{0}^{2\pi }f(x,y)e^{-imx}e^{-iny}dxdy
\end{equation*}
are the Fourier coefficients of the function $f$. The rectangular partial sums are defined as follows:
\begin{equation*}
S_{M,N} \left[ f, (x,y)\right] :=\sum_{m=-M }
^{M }
\sum_{n=-N }^N
\widehat{f}\left(
m,n\right) e^{imx}e^{iny},
\end{equation*}

In this paper we consider convergence of  {\bf  only rectangular partial sums} (convergence in the sense  of Pringsheim) of double Fourier series.

We denote by $C(T^{2})$ the space of continuous functions which are $2\pi $%
-periodic with respect to each variable with the norm
\begin{equation*}
\Vert f\Vert _{C}:=\sup_{x,y\in T^{2}}|f(x,y)|.
\end{equation*}

For the function $f$ defined on $T^2$ we denote by $f\left( x\pm 0,y\pm 0\right) $ the open coordinate quadrant limits (if exist) at the point $\left( x,y\right) $ and let $\frac14\sum f\left( x\pm0,y\pm0\right) $
be the arithmetic mean
\begin{equation}\label{limits}
\frac{1}{4}\{f\left( x+0,y+0\right) +f\left( x+0,y-0\right) +f\left(
x-0,y+0\right) +f\left( x-0,y-0\right) \}.
\end{equation}

The well known Dirichlet-Jordan theorem (see \cite{Zy}) states that the Fourier series of a function $f(x), \ x\in T$ of bounded variation converges
at every point $x$ to the value $\left[ f\left( x+0\right) +f\left(x-0\right) \right] /2$.
If $f$ is in addition continuous on $T$ the Fourier
series converges uniformly on $T$.

Hardy \cite{Ha} generalized the
Dirichlet-Jordan theorem to the double Fourier series. He proved that if  function $f(x,y)$ has bounded variation in the sense
of Hardy ($f\in BV$), then $S \left[ f\right] $ converges at any point $\left( x,y\right) $
to the value $\frac14\sum f\left( x\pm0,y\pm0\right) $. If $f$ is in addition continuous on
$T^{2}$ then $S \left[ f\right] $ converges uniformly on $T^{2}$.

\medskip
\textbf{Theorem S} (Sahakian \cite{Saha}).
{\it The Fourier series of a
function $f\left( x,y\right) \in HBV$ converges to $\frac14\sum f\left( x\pm0,y\pm0\right) $ at any point $\left( x,y\right) $, where the quadrant
limits (\ref{limits}) exist. The convergence is uniformly on any compact $K$, where
the function $f$ is continuous.}

\medskip
Theorem S was proved in \cite{Saha} under assumption that the function is continuous on some open set containing $K$
while Sargsyan noticed  in \cite {SG}, that the continuity of $f$ on the compact  $K$ is sufficient.

Analogs of Theorem S for higher dimensions can be found in \cite{Sab} and \cite {Bakh}.
Convergence of spherical and other partial sums of double Fourier series of functions of bounded $\Lambda$-variation was investigated in details by
Dyachenko (see \cite{D1}, \cite {D2} and references therein).

The first author \cite{GoEJA} has proved that if $f$ is
continuous function and has bounded partial $p$-variation $\left( f\in
PBV_{p}\right) $ for some $p\in [1,+\infty )$ then $S \left[ f\right] $ converges uniformly on $T^{2}$.
Moreover, the following is true

\textbf{Theorem G} ({Goginava} \cite{GoEJA}).
{\it Let $f\in C\left(
T^{2}\right) $ and
\begin{equation*}
\sum\limits_{n=1}^{\infty }\frac{\sqrt{v_{j}\left( n,f\right) }}{n^{3/2}}%
<\infty ,\,\,j=1,2.
\end{equation*}
Then $S \left[ f\right] $ converges uniformly on $T^{2}$.}
\medskip

Theorems 1, 2, Corollary 2 and Theorem S imply
\begin{theorem} \label{main}
Let $f\in P\Lambda BV$ with
\begin{equation*}
\sum\limits_{j=1}^{\infty }\frac{\lambda _{j}}{j^{2}}<\infty ,\quad \frac{%
\lambda _{j}}{j}\downarrow 0.
\end{equation*}
Then $S \left[ f\right] $ converges
to $\sum f\left( x,y\right) $ in any point $\left( x,y\right) $%
, where the quadrant limits (\ref{limits}) exist. The convergence is uniformly on any
compact $K$, where the function $f$ is continuous.
\end{theorem}
\begin{theorem}
Let $f\in B$ and
\begin{equation*}
\sum\limits_{n=1}^{\infty }\frac{\sqrt{v_{j}\left( n,f\right) }}{n^{3/2}}%
<\infty ,\quad j=1,2.
\end{equation*}
Then $S \left[ f\right] $  converges
to $\frac14\sum f\left( x\pm0,y\pm0\right) $ in any point $\left( x,y\right) $%
, where the quadrant limits (\ref{limits}) exist. The convergence is uniformly on any
compact $K$, where the function $f$ is continuous.
\end{theorem}

\begin{corollary}
Let $f\in B$ and $v_{1}\left( k,f\right) =O\left( k^{\alpha }\right)
,v_{2}\left( k,f\right) =O\left( k^{\beta }\right) ,\,0<\alpha ,\beta <1.$
Then $S \left[ f\right] $ converges
 to $\frac14\sum f\left(x\pm0,y\pm0\right) $ in any point $\left( x,y\right) $%
, where the quadrant limits (\ref{limits}) exist. The convergence is uniformly on any
compact $K$, where the function $f$ is continuous.
\end{corollary}

\begin{theorem}
Let $f\in PBV_{p},\,p\geq 1$. Then $S \left[ f\right] $ converges to $\frac14\sum f\left( x\pm0,y\pm0\right) $ in any point $%
\left( x,y\right) $, where the quadrant limits in (\ref{limits}) exist. The convergence
is uniformly on any compact $K$, where the function $f$ is continuous.
\end{theorem}

From Theorem \ref{main} follows that for any $\delta>0$ the  Fourier series of the function  $f\in P\left\{\frac n {\log^{1+\delta}n}\right\}BV$
converges to $\frac14\sum f\left( x\pm0,y\pm0\right) $  in any point $(x,y)$, where the quadrant limits  (\ref{limits}) exist. Moreover, one can not take here $\delta =0$ (see Theorem \ref{T6}).
It is interesting to compare this result with that obtained by M. Dyachenko and D. Waterman in \cite {DW}.

Dyachenko and Waterman \cite{DW} introduced another class of functions of generalized bounded variation. Denoting by $\Gamma $ the the set of finite collections of nonoverlapping rectangles $A_{k}:=\left[ \alpha _{k},\beta _{k}\right] \times \left[ \gamma _{k},\delta _{k}\right] \subset T^2$ we define
$$\Lambda ^{\ast }V\left( f\right) :=\sup_{\{A_k\}\in\Gamma }\sum\limits_{k}
\frac{\left\vert f\left( A_{k}\right) \right\vert }{\lambda _{k}}.
$$

\begin{definition}[Dyachenko, Waterman]
Let $f$ be a real function on $T^2:=\left[ 0,2\pi\right] \times \left[ 0,2\pi%
\right] $. We say that $f\in \Lambda ^{\ast }BV $ if
\begin{equation*}
\Lambda V(f):=\Lambda V_{1}(f)+\Lambda V_{2}(f)+\Lambda ^{\ast }V\left( f\right)<\infty .
\end{equation*}
\end{definition}
\begin{TDW}[\cite{DW}]
\label{TDW}
If $f\in \left\{\frac n {\log n}\right\}^*BV$,
then in any point $(x,y)$ the quadrant limits (\ref{limits}) exist and the double Fourier series of $f$ converges  to $\frac14\sum f\left(x\pm0,y\pm0\right) $.\\
Moreover, the sequence $\left\{\frac n {\log n}\right\}$ can not be replaced with any sequence  $\left\{\frac {n\alpha_n} {\log n}\right\}$, where $\alpha_n\to \infty$.
\end{TDW}

It is easy to show  {\rm(see\cite{DW})}, that
$\left\{\frac n {\log n}\right\}^*BV\subset HBV$, hence the convergence part of Theorem DW follows from Theorem S. It is essential that the condition $f\in \left\{\frac n {\log n}\right\}^*BV$ guaranties the  existence of quadrant limits.

The next theorem in particular shows that in Theorem S the condition
$HV_{1,2}(f)<\infty$ is necessary, i.e the boundedness of partial harmonic variation is not sufficient for the convergence of Fourier series of continuous function.
\begin{theorem}\label{T6}
Let $\Lambda =\{\lambda _{n}=n\gamma_n\}$ where $\gamma_{n}$ is an decreasing sequence satisfying
(\ref{T1-10}) and (\ref{T1-11}).
Then there exists a continuous function $f\in P\Lambda BV$ with Fourier series that diverges at $(0,0)$.
\end{theorem}
We need the following simple lemma that easily follows from Lemma \ref{L1}.
\begin {lemma}\label{Lem}
Let the function $g(t)$ be defined on $T$ and
$$
0=t_1<t_2<\ldots<t_{2m}=2\pi.
$$
Suppose $g$ is increasing on $[t_i,t_{i+1}]$
if $\,i$ is odd and is decreasing, if $\,i$ is even. If
$$
\left|g(t_{i+1})-g(t_{i})\right|>
\left|g(t_{i+2})-g(t_{i+1})\right|, \quad i=1,2,\ldots,2m-2,
$$ then
$$
\Lambda BV(g)=\sum_{i=1}^{2m-1}\frac{\left|g(t_{i+1})-g(t_{i})\right|}
{\lambda_i},
$$
for any sequence $\Lambda=\{\lambda_n\}$ satisfying (\ref{Lambda}).
\end {lemma}
\begin{proof}[Proof of Theorem \ref{T6}]
It is not hard to see, that for any sequence $\Lambda=\{\lambda_n\}$ satisfying (\ref{Lambda}) the class $C(T^2)\cap P\Lambda BV$ is a Banach space with the norm
$$
\|f\|_{P\Lambda BV}:=\|f\|_C+P\Lambda V(f).
$$
Let $\Lambda=\{\lambda_n\}$ be as in Theorem \ref{T6} and denote
$$A_{i,j}=\left[\frac{\pi i}{N+1/2}, \ \frac{\pi (i+1)}{N+1/2}\right)\times
\left[\frac{\pi j}{N+1/2}, \ \frac{\pi (j+1)}{N+1/2},\right).
$$
Define $t_j$ and $m_j$ as in (\ref{tj}) and consider the function
\begin{equation}\label{T6-1}
f_N(x,y)=\sum_{(i,j)\in W}t_j\,
\chi_{A_{i,j}}(x,y)\sin\left(N+\frac12\right)x
\cdot\sin\left(N+\frac12\right)y,
\end{equation}
where $\chi_A(x,y)$ is the characteristic function of the set $A\subset T^2$ and
$$
W:=\left\{(i,j)\,:\,j<i<j+m_j,\quad 1\leq j<N_\delta\right\},\qquad N_\delta=\left(\frac N2\right)^{\frac1{1+\delta}}.
$$
Each summand in the sum (\ref{T6-1}) is continuous on the rectangle $A_{i,j}$ and vanishes on its boundary, hence  $f_N\in C(T^2)$.

Next, in view of Lemma \ref{Lem}, using the same arguments as in the proof of (\ref{4}) and (\ref{5}), we get
$$
\Lambda V_1(f_N)\leq1,\qquad
\Lambda V_2(f_N)\leq1.
$$
Hence $f_N\in P\Lambda BV$ and
\begin{equation}\label{16}
\|f_N\|_{P\Lambda V}\leq 3,\quad  N=1,2,\ldots.
\end{equation}
 Observe that $N_\delta<N$ and $j+m_j<N$, if $j<N_\delta$, hence $A_{i,j}\subset T^2$, if $(i,j)\in W$.
 Taking onto account (\ref{tj}) and (\ref{tjl}), for the square partial sum of the Fourier series of $f_N$ at $(0,0)$ we get
\begin{eqnarray}\label{17}
&\pi& \cdot S_{N,N}[f_N,(0,0)]=\int_{T^2}f_N(x,y)D_N(x)D_N(y)dxdy
\notag\\&=&
\sum_{(i,j)\in W}{t_j}
\int_{A_{i,j}}\frac{\sin^2\left(N+\frac12\right)x
\cdot\sin^2\left(N+\frac12\right)y}{4sin\frac x2sin\frac y2}dxdy\\
&\geq &c
\sum_{j=1}^{N_\delta}\frac {t_j}j
\sum\limits_{i=j+1}^{j+m_j}\frac1i
\geq
c\sum\limits_{j=1}^{N_\delta}
\frac {t_j}j\log(j+m_j)
\geq
c\sum\limits_{j=1}^{N_\delta}
\frac {\gamma_j}j\to\infty .
\notag
\end{eqnarray}
as $N\to\infty$, where $c$ is an absolute constant.

Applying the Banach-Steinhaus Theorem, from (\ref{16}) and (\ref{17}) we obtain that there exists a continuous function $f\in P\Lambda BV$ such that
$$
\sup_N |S_{N,N}[f,(0,0)]|=\infty.
$$
\end{proof}

\end{document}